\documentclass[a4paper]{article}
\usepackage[colorlinks=true, hypertexnames=false]{hyperref}
\usepackage{enumitem}
\usepackage{mathtools}
\usepackage{amssymb, amsmath}
\usepackage[amsthm,thmmarks]{ntheorem}
\usepackage{mathrsfs}

{\theoremstyle{plain}
\newtheorem{proposition}{Proposition}[section]
\newtheorem{lemma}[proposition]{Lemma}
\newtheorem{theorem}[proposition]{Theorem}
\newtheorem{corollary}[proposition]{Corollary}
}

\renewcommand{\phi}{\varphi}
\renewcommand{\epsilon}{\varepsilon}
\renewcommand{\theta}{\vartheta}

\newcommand{\deph}{\textbf}
\newcommand{\set}[2]{\left\{ #1 \,\middle|\, #2 \right\}}

\newcommand{\lh}{{\ell_{\mathrm{H}}}}

\newcommand{\lj}{{\ell_{\mathrm{J}}}}
\newcommand{\lr}{{\ell_{\mathrm{r}}}}

\newcommand{\alev}[1]{\,\,\left[#1\right]}

\newcommand{\bs}{\boldsymbol}
\newcommand{\mup}[2]{\left(#1\right)_{#2}}

\newcommand{\GLinG}{\mathcal{GL}}

\DeclareMathOperator{\GL}{GL}
\DeclareMathOperator{\SL}{SL}

\DeclareMathOperator{\diam}{diam}
\DeclareMathOperator{\rank}{rk}
\DeclareMathOperator{\id}{id}
\DeclareMathOperator{\chr}{char}
\DeclareMathOperator{\fin}{fin}

\title{Properties of Linearly Sofic Groups}
\author{Abel Stolz}
\date{\today}

\begin{document}
\maketitle
\begin{abstract}
 We consider (projectively) linearly sofic groups, i.e. groups which can be approximated using (projective) matrices over arbitrary fields, as a generalization of sofic groups. We generalize known results for sofic groups and groups which can be approximated with complex matrices, including the fact that free products of linearly sofic groups (using a fixed field) are linearly sofic.
\end{abstract}

\section{Introduction}
Recently approximation of groups using matrices has gained interest. This approach enriches the field of group approximation, dominated by the investigation of sofic groups and to a certain amount hyperlinear groups. Notably Arzhantseva and P\u{a}unescu in \cite{arzhantsevapaunescu13} studied groups that can be approximated using complex matrices, where in contrast to hyperlinear groups the metric used measures the rank of differences of matrices. Since this metric does not depend on the underlying field, the same approach works for matrix groups over arbitrary fields. We call groups allowing for this kind of approximation \emph{linearly sofic groups} (similarly to \cite{arzhantsevapaunescu13}). The class of linearly sofic groups is on the one hand an interesting candidate to produce an example of a non-sofic group, and on the other hand could add another point of view to the theory of sofic groups, assuming it will be proved to be contained in the class of sofic groups.

We follow some of the work in \cite{arzhantsevapaunescu13} closely, which in turn mimics theorems known to hold for sofic groups since the work of Elek and Szab\'o \cite{elekszabo05} and \cite{elekszabo06}. We thereby reproduce various results in somewhat broader generality. After initially introducing two means of approximation, using the above-mentioned rank difference metric and a projective version thereof, we show eventually that both approaches are equivalent. A new result is Theorem~\ref{thm:FreeProductsOfLinearlySoficGroups}, proving that free products of linearly sofic groups (using a fixed field) are linearly sofic. We still don't know whether this generalizes to free products amalgamated over amenable groups, as is the case for sofic groups. (Confer \cite{collinsdykema11}, \cite{elekszabo11} and \cite{paunescu11}.)

This article is organized as follows: In Section~\ref{sec:LengthFunctions} we introduce the basic definitions centering around length functions and the notion of group approximation.

In Section~\ref{sec:Amplification} we explain the important tool of amplification, which will be essential for the remainder of the article.

Section~\ref{sec:GroupApproximationWithMatrices} develops on group approximation as explained in Section~\ref{sec:LengthFunctions}, where the groups used for approximation are general linear groups over different fields. An important goal of investigation is to clarify the role of the underlying field, i.e. essentially the role of its characteristic. Also the relation with approximation in projective linear groups is examined.

In Section~\ref{sec:LinearlySoficAndProjectivelyLinearlySoficGroups} we show that the class of $K$-sofic groups (where $K$ is a field) satisfies certain permanence properties.

The contents of this article are part of the author's PhD thesis, handed in to the University of Leipzig on April, 23. 2013.

\textbf{Acknowledgements.}
I want to thank Professor Andreas Thom for his support and helpful guidance during the time I wrote my PhD thesis.

\section{Length functions and linear group approximation}\label{sec:LengthFunctions}
Let $G$ be a group. A function 
\[\ell\colon G\rightarrow [0,\infty[\]
 is called a \deph{pseudo length function} on $G$ if for all $g,h\in G$
\begin{enumerate}[label={\rm LF\arabic*}]
 \item $\ell(1)\geq 0$, $1\in G$,\label{def:LengthFunctionDefinite}
 \item $\ell(g)=\ell(g^{-1})$,\label{def:LengthFunctionSymmetry}
 \item $\ell(gh)\leq \ell(g)+\ell(h)$.\label{def:LengthFunctionTriangle}
\end{enumerate}
If moreover $\ell(hgh^{-1})=\ell(g)$ holds, then we call $\ell$ \deph{invariant}. If  $\ell(g)=0$ if and only if $g=1$, then $\ell$ is a \deph{length function}. The \deph{diameter} $\diam(G)$ of a group $G$ with pseudo length function $\ell$ is defined as $\sup_{g\in G}\ell(g)$.

A standard example of a length function is the \deph{Hamming length}
\[\lh(\pi)\coloneqq \frac{|\set{i\in [n]}{\pi(i)\neq i}|}{n}\]
for permutations $\pi$ in the symmetric group $S_n$ (where $[n]=\{1,\ldots, n\}$).

Given a finite-dimensional vector space $V$ we write $\GL(V)$ for all bijective linear transformations of $V$ and $\SL(V)$ for all linear transformations of $V$ of determinant $1$. When $V=K^n$ for some field $K$ we write $\GL_n(K)\coloneqq \GL(V)$ and $\SL_n(K)\coloneqq \SL(V)$. We will think of elements in $\GL_n(K)$ as matrices corresponding to the standard basis in $K^n$.

 If $V$ is a vector space over a field $K$ we will write $1$ for the identical mapping $V\rightarrow V$ and write simply $\alpha$ for the mapping $\alpha \cdot 1$, where $\alpha\in K$.

We define the \deph{rank length} on $\GL(V)$, where $V$ is an $n$-dimensional vector space over the field $K$ by
\[\lr(g)\coloneqq\frac{\rank(1-g)}{n}\]
We will also need the \deph{Jordan length}:
\[\lj(g)=\inf_{\alpha\in K^{\times}}\frac{\rank(\alpha - g)}{n}.\]

Confer \cite{stolzthom13} for a more detailed exposition of length functions, in particular for the verification that $\lr$ is an invariant length function and $\lj$ is an invariant pseudo length function.

Let $G_i$ be groups with invariant pseudo length functions $\ell_i$ of bounded diameter for all $i\in I$, where $I$ is an arbitrary index set. Let $\mathfrak{u}$ be a (non-principal) ultrafilter in $I$. We define the subset $N$ of the direct product $\prod_{i\in I}G_i$ to be
\[N\coloneqq \set{(g_i)_{i\in I}\in \prod_{i\in I}G_i}{\lim_{\mathfrak{u}}\ell_i(g_i)=0}.\]
Here $\lim_{\mathfrak{u}}\ell_i(g_i)$ is the \deph{ultralimit} of the numbers $\ell_i(g_i)$, i.e. the unique real number $x$ such that for all $\epsilon>0$ the set $\set{i\in I}{|x-\ell_i(g_i)|<\epsilon}$ is in $\mathfrak{u}$. Then the properties of $\ell_i$ imply that $N$ is a normal subgroup and we call the group
\[\mup{\prod_{i\in I}G_i}{\mathfrak{u}}\coloneqq \left.\prod_{i\in I}G_i\middle/N\right.\]
the \deph{metric ultraproduct} of the groups $G_i$. We write $P(i)\alev{\mathfrak{u}}$ if a property $P$ holds for $\mathfrak{u}$ almost all $i$.

Let $\mathscr{G}$ be a class of groups, each of which comes equipped with a pseudo length function $\ell$. Then a group $\Gamma$ is said to have the $\mathscr{G}$\deph{-approximation property} if for all $g\in \Gamma\setminus\{1\}$ there is $\delta_g>0$ such that for any $\epsilon>0$ and any finite subset $E\subset \Gamma$ there is a group $G\in\mathscr{G}$ and a mapping $\phi\colon \Gamma\rightarrow G$ such that
\begin{enumerate}[label={\rm AH\arabic*}]
 \item $\ell(\phi(1))\leq \epsilon$,\label{def:ApprPropOne}
 \item $\ell(\phi(g))\geq \delta_g$ for all $g\in E\setminus\{1\}$,\label{def:ApprPropSingleElement}
 \item $\ell(\phi(g)\phi(h)\phi(gh)^{-1})\leq \epsilon$ for all $g,h\in E$.\label{def:ApprPropMultiplication}
\end{enumerate}
The kind of mapping in this definition is called $(E,\epsilon)$\deph{-homomorphism} or less explicit \deph{almost homomorphism}. Note that almost homomorphisms depend not only on $E$ and $\epsilon$ but also on the distribution of numbers $\delta_g$, $g\in \Gamma$. 
Another method of approximating groups which is used often we will adress here as the \deph{discrete} $\mathscr{G}$\deph{-approximation property}. A group has this property if we replace the constants $\delta_g$ in the above definition by a common constant $\delta$ which does only depend on the whole of $\Gamma$.
The \deph{strong discrete} $\mathscr{G}$\deph{-approximation property} demands $\ell(\phi(g))\geq \diam(G)-\epsilon$ for all $g\in E\setminus\{1\}$ instead of \ref{def:ApprPropSingleElement}. Such $\phi$ is called a \deph{strong} almost homomorphism. Of course this only makes sense if the groups $G\in\mathscr{G}$ have finite diameter.

We proceed by exhibiting the connection between group approximation and metric ultraproducts. The following fundamental theorem is a generalization of Theorem~1 in \cite{elekszabo05} and the proof is the same. Confer also \cite{thom12}, Proposition~1.8.

\begin{theorem}\label{thm:ApproximationPropertyAndMetricUltraproducts}
Let $\Gamma$ be a group. Then $\Gamma$ has the $\mathscr{G}$-approximation property if and only if there is a suitable index set $I$ and an ultrafilter $\mathfrak{u}$ in $I$ such that $\Gamma$ can be embedded into a metric ultraproduct $\mup{\bs{G}}{\mathfrak{u}}\coloneqq \mup{\prod_{i\in I} G_i}{\mathfrak{u}}$ with groups $G_i\in\mathscr{G}$. The set $I$ can be chosen to have cardinality not exceeding the cardinality of $\Gamma$.

Moreover $\Gamma$ has the discrete $\mathscr{G}$-approximation property if and only if it embeds into $\mup{\bs{G}}{\mathfrak{u}}$ as a discrete subgroup.
\end{theorem}

Note that the very definition of $\mathscr{G}$-approximation immediately implies that a group $\Gamma$ has the $\mathscr{G}$-approximation property if and only if every finitely generated subgroup does. Hence it often suffices to study countable groups with the $\mathscr{G}$-approximation property. The same is true for the discrete and strong discrete approximation property.

To make further investigation a bit more pleasant, we relax the conditions characterizing almost homomorphisms. Condition~\ref{def:ApprPropOne} is not necessary in the definition of almost homomorphisms, provided $1\in E$:

\begin{proposition}\label{prop:RelaxedAlmostHomomorphisms}
If $\Gamma$ has the $\mathscr{G}$-approximation property, then for all finite $E\subset \Gamma$ and $\epsilon>0$ there is an $(E,\epsilon)$-homomorphism $\phi$ such that $\phi(1)=1$. Furthermore for $g\in E$ the length $\ell(\phi(g))$ is as large as we can expect of any almost homomorphism $\Gamma\rightarrow G\in\mathscr{G}$.
\end{proposition}

\begin{proof}
 Let $E\subset \Gamma$ be finite and $\epsilon>0$, and assume $1\in E$. Let $\psi\colon \Gamma\rightarrow G$ be a mapping satisfying \ref{def:ApprPropSingleElement} and \ref{def:ApprPropMultiplication}. Then
\[\ell(\psi(1))=\ell(\psi(1)\cdot \psi(1)\psi(1\cdot 1)^{-1})\leq \epsilon\]
follows.

Now assume $\psi$ is an $(E,\frac{1}{2}\epsilon)$-homomorphism. We define $\phi$ to take the same values as $\psi$ does, except $\phi(1)\coloneqq 1$. It suffices to show $\ell(\phi(g)\phi(h)\phi(gh)^{-1})\leq \epsilon$ for $g,h\in E\cup\{1\}$ and $gh\in E^2\cup\{1\}$. The cases to check are (by symmetry) without loss of generality $g=h^{-1}$, $g\neq 1$, and $g\neq 1$, $h= 1$. In the first case
\begin{align*}
\ell(\phi(g)\phi(g^{-1}))&=\ell(\psi(g)\psi(g^{-1})\psi(1)^{-1}\psi(1))\\
&\leq \ell(\psi(g)\psi(g^{-1})\psi(gg^{-1})^{-1})+\ell(\psi(1))\\
&\leq \tfrac{1}{2}\epsilon+\tfrac{1}{2}\epsilon=\epsilon.
\end{align*}
In the second case we are left to check $\ell(\phi(g)\phi(g)^{-1})\leq \epsilon$. This is true, because
\begin{align*}
 \ell(\phi(g)\phi(g)^{-1})&=\ell(\psi(g)\psi(g^{-1})\psi(1)^{-1}\psi(1)\psi(g^{-1})^{-1}\psi(g)^{-1})\\
&\leq \ell(\psi(g)\psi(g^{-1})\psi(1)^{-1})+\ell(\psi(1)\psi(g^{-1})^{-1}\psi(g)^{-1})\\
&\leq \tfrac{1}{2}\epsilon+\tfrac{1}{2}\epsilon=\epsilon,
\end{align*}
and the proof is complete.
\end{proof}

Let $\GLinG$ denote the class of all general linear groups over arbitrary fields and $\GLinG(K)$ the class of general linear groups over a fixed field $K$. We shall call groups with the $\GLinG$-approximation property using the rank length \deph{linearly sofic} groups. If a group has the $\GLinG(K)$-approximation property for a fixed field $K$ we will call it $K$\deph{-sofic}. When instead approximation is done using the Jordan length we will speak of \deph{projectively linearly sofic} and \deph{projectively} $K$\deph{-sofic} groups, respectively. 

Note that the \emph{linear sofic groups} introduced in \cite{arzhantsevapaunescu13} are $\mathbb{C}$-sofic groups in the above sense.

\section{Amplification properties}\label{sec:Amplification}
In the definition of group approximation whether a group $\Gamma$ has the \emph{(strong) discrete} approximation property a priori depends on $\Gamma$. In certain classes $\mathscr{G}$ we can enforce the (strong) discrete $\mathscr{G}$-approximation property for every group having the $\mathscr{G}$-approximation property: A class of groups $\mathscr{G}$ has the \deph{amplification property} if there exists $\delta>0$ such that for any group $\Gamma$ and for all $\epsilon>0$ there exists $\epsilon'>0$ such that the following holds: Whenever $\phi\colon \Gamma\rightarrow G$ is an $(E,\epsilon')$-homomorphism into $G\in\mathscr{G}$, then there is $H\in\mathscr{G}$ and a mapping $\iota\colon G\rightarrow H$ such that $\iota\circ \phi\colon \Gamma\rightarrow H$ is an $(E, \epsilon)$-homomorphism with the additional property that $\ell(\iota\circ \phi(g))\geq \delta$ for all $g\in E\setminus\{1\}$. If there exist $H$ and $\iota$ such that $\iota \circ \phi$ satisfies $\ell(\iota\circ \phi(g))\geq \diam(H)-\epsilon$ we say that $\mathscr{G}$ has the \deph{strong amplification property}. (Note that $\delta$ is no longer needed in the second definition.)

Now the next proposition follows directly.

\begin{proposition}\label{prop:ConsequenceOfAmplificationProperty}
If $\mathscr{G}$ has the (strong) amplification property, then any group $\Gamma$ has the (strong) discrete $\mathscr{G}$-approximation property if and only if it has the $\mathscr{G}$-approximation property.
\end{proposition}

The maybe best known example of amplification are the symmetric groups. The proof appears e.\,g. in \cite{elekszabo05}, proof of Theorem~1 or in \cite{pestov08}, proof of Theorem~3.5.

\begin{proposition}
The class $\mathcal{S}$ of symmetric groups with the Hamming length has the strong amplification property.
\end{proposition}

Consider the following statements about approximation of subgroups, inverse limits and direct products of groups.

\begin{proposition}\label{prop:SubgroupsAndInverseLimitsPermanence}
Let $\mathscr{G}$ be a class of groups with invariant pseudo length functions. Then the class of groups with the $\mathscr{G}$-approximation property is closed under taking subgroups and inverse limits. The same is true for the discrete and strong discrete $\mathscr{G}$-approximation property.
\end{proposition}

\begin{proof}
The statement concerning subgroups is obvious.

Assume that $\Gamma$ is the inverse limit of groups $\Gamma_i$ with projection morphisms $\pi_i^j\colon \Gamma_i\rightarrow \Gamma_j$ and $\pi_i\colon \Gamma\rightarrow \Gamma_i$, where $i$, $j$ are from a directed set $I$. Then $\Gamma$  can be identified with the set of vectors $(g_i)_{i\in I}$ in $\prod_{i\in I}\Gamma_i$ such that $\pi_i^j(g_i)=g_j$ if $j\leq i$. If $E$ is a finite subset of $\Gamma$ there is an index $i$ such that $\pi_i(g)\neq \pi_i(h)$ for all $g, h\in E^2$. By assumption for any $\epsilon>0$ there is a $(\pi_i(E),\epsilon)$-homomorphism $\phi\colon \Gamma_i\rightarrow G$ for some $G\in\mathscr{G}$. 
By the choice of $i$, $\phi\circ \pi_i$ is well defined on $E^2$. Then some arbitrary extension of $\phi\circ\pi_i$ to the rest of $\Gamma$ is automatically an $(E,\epsilon)$-homomorphism. The same argument works for the stronger $\mathscr{G}$-approximation properties in the claim.
\end{proof}

\begin{proposition}\label{prop:DirectProductsPermanence}
Suppose that $\mathscr{G}$ has the following property: For every finite direct product $G_1\times \ldots \times G_k$ of groups $G_i\in\mathscr{G}$ there are weights $w_1,\ldots,w_k\in [0,\infty[$, a group $H\in \mathscr{G}$ and an isometric embedding $G_1\times \ldots \times G_k\rightarrow H$, where we use the pseudo length function
\[\ell((g_1,\ldots,g_k))\coloneqq \frac{\sum_{i=1}^k w_i\ell(g_i)}{\sum_{i=1}^kw_i}\]
on $G_1\times\ldots\times G_k$. Then the class of groups with the $\mathscr{G}$-approximation property is closed under taking direct products. The same holds for the discrete and strong discrete $\mathscr{G}$-approximation property.
\end{proposition}

\begin{proof}
Because being approximated with groups in $\mathscr{G}$ is a local property, it clearly suffices to consider finitary direct products. Now the additional assumption on $\mathscr{G}$ implies the claim immediately. The argument also works for the discrete ond strong discrete approximation property.
\end{proof}

Certainly statements as above for other group theoretical constructions would be a great thing to have. Unfortunately under very general assumptions very little can be done. An example where the amplification property is needed is the following.

\begin{proposition}\label{prop:DirectLimitsPermanence}
If $\mathscr{G}$ has the amplification property, then the class of groups with the $\mathscr{G}$-approximation property is closed under taking direct limits.
\end{proposition}

\begin{proof}
Let $\Gamma$ be a direct limit of groups $\Gamma_i$, which can be approximated in $\mathscr{G}$. Let $\epsilon>0$ and $E\subset \Gamma$ be a finite subset. Because $E$ is finite, eventually $E^2\subset \Gamma_i$ holds. Thus we find an $(E,\epsilon)$-homorphism $\phi\colon \Gamma_i\rightarrow G$, where $G\in \mathscr{G}$, which can be extended arbitrarily to an almost homomorphism defined on $\Gamma$. Then $\ell(\phi(g))\geq \delta_{g,i}>0$. Since a priori $\delta_{g,i}$ depends on $i$, we need the amplification property to ensure $\delta_{g,i}\geq \delta_g$ for constants $\delta_g$ not depending on $i$.
\end{proof}

In \cite{arzhantsevapaunescu13} Arzhantseva and P\u{a}unescu showed that $\set{\GL_n(\mathbb{C})}{n\in \mathbb{N}}$ with the rank length has the amplification property. In fact the proof does not use particular properties of the complex numbers apart from characteristic zero and it could be generalized to arbitrary fields of characteristic zero.
We will modify the method of proof from \cite{arzhantsevapaunescu13}, Section~5, to work in any characteristic and to show the amplification property also when working with the Jordan length.

Let $K$ be an algebraically closed field. Then the Jordan decomposition for matrices in $\GL_n(K)$ exists. In particular in algebraically closed fields every square matrix is conjugate to a matrix in Jordan normal form.

Over any field we use the notation $J(\alpha,s)$ for $s\times s$-Jordan matrices with eigenvalue $\alpha$. Every matrix $A\in \GL_n(K)$ is in $\GL_n(\overline{K})$ conjugate to a matrix $A'$ in Jordan normal form, where $\overline{K}$ is the algebraic closure of $K$. We write $\iota_\alpha(A)$ for the number of Jordan blocks $J(\alpha,1)$ of $A'$ divided by $n$. Furthermore let
\[\iota(A)\coloneqq \sup_{\alpha\in K^\times}\iota_\alpha(A).\]

\begin{proposition}\label{prop:LengthFunctionsAndIotaValuesEstimates}
 Let $K$ be a field and $A$ in $\GL_n(K)$. Then 
\[\tfrac{1}{2}(1-\iota_1(A))\leq \lr(A)\leq 1-\iota_1(A),\quad\quad  \tfrac{1}{2}(1-\iota(A))\leq \lj(A)\leq 1-\iota(A).\]
\end{proposition}

\begin{proof}
Let $A'$ be a matrix in Jordan normal form conjugate to $A$ in $\GL_n(\overline{K})$. If this matrix has $k$ Jordan blocks $J(1,1)$, then 
\[\lr(A)=\frac{\rank(1-A)}{n}\leq \frac{n-k}{n}=1-\iota_1(A).\]
Since the remaining Jordan blocks do not have eigenvalues equal to $1$ or are of size strictly larger than $1$, also
\[\lr(A)=\frac{\rank(1-A)}{n}\geq \frac{n-k}{2n}=\tfrac{1}{2}(1-\iota_1(A)).\]
The claimed inequalities for $\iota(A)$ and $\lj(A)$ follow from the above result, since \[\lj(A)=\inf_{\alpha\in K^\times}\lr(\alpha A)\]
 and 
\[\inf_{\alpha\in K^\times}(1-\iota_1(\alpha A))=1-\sup_{\alpha\in K^\times}\iota_\alpha(A)=1-\iota(A).\]
\end{proof}

\begin{theorem}\label{thm:FewJordanBlocksInTensorProductMatrices}
 Let $K$ be an algebraically closed field and $J(\alpha,s)$, $J(\beta,t)$ two Jordan matrices with eigenvalues $\alpha,\beta\in K^\times$ and $s\leq t$. Then the Jordan normal form of $J(\alpha,s)\otimes J(\beta,t)$ has $s$ Jordan blocks.
\end{theorem}

\begin{proof}
 In characteristic $0$ the claim follows from Corollary~2.2.11 in \cite{iimaiwamatsu09}, in positive characteristic from Theorem~2.2.2, ibid.
\end{proof}

\begin{lemma}\label{lem:SizeOfJordanBlockWithInseparableEigenvalue}
 Let $A$ be a matrix in $\GL_n(K)$ and $\alpha$ an eigenvalues of $A$. Suppose that in a matrix $A'$ in Jordan normal form, obtained from $A$ over $\overline{K}$, the Jordan block corresponding to $\alpha$ has size $s$. If the extension $K(\alpha)/K$ is inseparable, then $s=p^k$, where $p$ is the characteristic of $K$ and $k>0$.
\end{lemma}

\begin{proof}
 First of all we note that for $K(\alpha)/K$ to be inseparable $K$ necessarily has to be of positive characteristic. Let $f$ be the minimal polynomial of $\alpha$ over $K$. Let $\alpha_1,\ldots,\alpha_m$ be the roots of $f$, where $\alpha_i=\alpha_j$ if and only if $i=j$. By \cite{lang02}, Chapter~V, Proposition~6.1 \[f=(x-\alpha_1)^{p^k}\cdot \ldots \cdot (x-\alpha_m)^{p^k}\]
for a natural number $k$, since $f$ is inseparable. Moreover $f$ divides the minimal polynomial $\mu_A$ of $A$, because every root of $f$ is a root of $\mu_A$. The remaining factor $g$ such that $\mu_A=f\cdot g$ has only roots different from the roots of $f$. Hence the multiplicity of $\alpha$ as a root of $\mu_A$ is $p^k$. Therefore $\dim\ker(A-\alpha)^j>\dim\ker(A-\alpha)^{j-1}$ if and only if $j\in \{1,\ldots, p^k\}$. This means that the Jordan block corresponding to $\alpha$ has size $p^k$.
\end{proof}

\begin{lemma}\label{lem:MultiplyingLargerAndSmallerNumbers}
 Let $x'\geq x\geq 0$, $y'\geq y\geq 0$ be real numbers. Then
\[x'y+y'x\leq x'y'+xy.\]
\end{lemma}

\begin{proof}
 We calculate
\begin{align*}
2(x'y+y'x) &= x'(y+y'-y')+(x'+x-x)y+y'(x+x'-x')+(y'+y-y)x\\
&=x'y' +x'(y-y')+xy+(x'-x)y\\
&\quad\quad +y'x'+y'(x-x')+yx+(y'-y)x\\
&=2x'y'+2xy +(x'-x)(y-y')+(y'-y)(x-x')\\
&\leq 2(x'y'+xy)
\end{align*}
to complete the proof.
\end{proof}

\begin{lemma}\label{lem:OneByOneJordanBlocksOfTensorProducts}
Let $K$ be a field and $A\in \GL_n(K)$, $B\in \GL_m(K)$. Then
\[\iota(A\otimes B)\leq \iota(A)\iota(B)+(1-\iota(A))(1-\iota(B)).\]
If $\iota(A)\leq \frac{1}{2}$ and $\iota(B)\leq \frac{1}{2}$, then $\iota(A\otimes B)\leq \frac{1}{2}$.
\end{lemma}

\begin{proof}
 We work with $A$ embedded in $\GL_n(\overline{K})$ and $B$ in $\GL_m(\overline{K})$. Let $A'$, $B'$ be matrices in Jordan normal form corresponding to $A$ and $B$, respectively. Then it is clear that the Jordan normal form of $A\otimes B$ is the same as the one of $A'\otimes B'$, since using conjugation to compute the Jordan normal form commutes with the tensor product. To obtain the Jordan normal form of $A'\otimes B'$ it is clearly sufficient to compute the Jordan normal forms of $J(\alpha,s)\otimes J(\beta,t)$ for all combinations of Jordan blocks $J(\alpha,s)$ of $A'$ and $J(\beta,t)$ of $B'$.

Since $J(\alpha,1)\otimes J(\beta, t)$ equals $\alpha J(\beta,t)$, we know that on the one hand two Jordan blocks of size $1$ yield a Jordan block of size $1$ in the Jordan normal form of $A\otimes B$. On the other hand a Jordan block of size $1$ and a larger one cannot produce a Jordan block of size $1$.
 Moreover two Jordan blocks $J(\alpha,s)$ and $J(\beta,t)$, where $1<s\leq t$, can be responsible for at most $s-1$ Jordan blocks of size $1$, by Theorem~\ref{thm:FewJordanBlocksInTensorProductMatrices}. Assume $\alpha$ is an eigenvalue of $A$ or $B$ such that $K(\alpha)/K$ is inseparable. Then by Lemma~\ref{lem:SizeOfJordanBlockWithInseparableEigenvalue} $\alpha$ corresponds to a Jordan block of size larger than or equal to the characteristic of $K$, in particular strictly larger than $1$. Denote the separable closure of $K$ by $\overline{K}_{\rm s}$ and let $\kappa(A)\coloneqq \sum_{\alpha\in \overline{K}_{\rm s}^\times}\iota_\alpha(A)$. If $\gamma\in K^\times $ such that $\iota_\gamma(A\otimes B)=\iota(A\otimes B)$ we can deduce
\begin{align*}\iota(A\otimes B)&\leq \sum_{\alpha\in K^\times}\iota_\alpha(A)\iota_{\alpha^{-1}\gamma}(B)+\sum_{\alpha\in\overline{K}_{\rm s}^\times\setminus K^\times}\iota_\alpha(A)\iota_{\alpha^{-1}\gamma}(B)\\
&\quad\quad +\tfrac{1}{2}(1-\kappa(A))(1-\kappa(B)).
\end{align*}
Here the splitting in sums over $K^\times$ and $\overline{K}_{\rm s}^\times\setminus K^\times$ is possible, because $K^\times$ is a subgroup of $\overline{K}_{\rm s}^\times$.

Let $\lambda$, $\delta$ be in $K^\times$ such that $\iota_\lambda(A)=\iota(A)$ and $\iota_\delta(B)=\iota(B)$. Then
\begin{align*}
\sum_{\alpha\in K^\times}\iota_\alpha(A)\iota_{\alpha^{-1}\gamma}(B)&= \iota_\lambda(A)\iota_{\lambda^{-1}\gamma}(B)+\iota_{\delta^{-1}\gamma}(A)\iota_\delta(B)\\
&\quad\quad  +\sum_{ \lambda ,\delta^{-1}\gamma\neq \alpha\in K^\times}\iota_\alpha(A)\iota_{\alpha^{-1}\gamma}(B)\\
&\leq \iota_\lambda(A)\iota_\delta(B)+\iota_{\delta^{-1}\gamma}(A)\iota_{\lambda^{-1}\gamma}(B)\\
&\quad\quad  +\sum_{\lambda ,\delta^{-1}\gamma\neq \alpha\in K^\times}\iota_\alpha(A)\iota_{\alpha^{-1}\gamma}(B)\\
&\leq \iota_\lambda(A)\iota_\delta(B)+\sum_{\lambda\neq \alpha\in K^\times}\iota_\alpha(A)\sum_{\delta\neq \beta\in K^\times}\iota_\beta(B),
\end{align*}
where we used Lemma~\ref{lem:MultiplyingLargerAndSmallerNumbers}.
By the choice of $\gamma$, $\lambda$ and $\delta$, and the preceding estimate of $\iota(A\otimes B)$ we arrive at
\[\iota(A\otimes B)\leq \iota(A)\iota(B)+(1-\iota(A))(1-\iota(B)),\]
which proves the first claim.

Now assume $\iota(A),\iota(B)\leq \frac{1}{2}$. If the eigenvalues of $A$ in $K$ are $\lambda_i$ such that $\iota_{\lambda_1}(A)\geq \iota_{\lambda_2}(A)\geq \ldots $ and the eigenvalues of $B$ in $K$ are $\delta_i$ such that $\iota_{\delta_1}(B)\geq \iota_{\delta_2}(B)\geq \ldots $, then we can proceed inductively from
\begin{align*}
\sum_{\alpha\in K^\times}\iota_\alpha(A)\iota_{\alpha^{-1}\gamma}(B)&\leq \iota_{\lambda_1}(A)\iota_{\delta_1}(B)+\iota_{\delta_1^{-1}\gamma}(A)\iota_{\lambda_1^{-1}\gamma}(B)\\
&\quad\quad +\sum_{\lambda_1 ,\delta_1^{-1}\gamma\neq \alpha\in K^\times}\iota_\alpha(A)\iota_{\alpha^{-1}\gamma}(B) 
\end{align*}
to obtain
\[\sum_{\alpha\in K^\times}\iota_\alpha(A)\iota_{\alpha^{-1}\gamma}(B)\leq \sum_{i}\iota_{\lambda_i}(A)\iota_{\delta_i}(B).\]

If $K(\alpha)/K$ is separable, then $\alpha$ has at least one Galois conjugate eigenvalue. This implies
\[\sum_{\alpha\in\overline{K}_{\rm s}^\times\setminus K^\times}\iota_\alpha(A)\iota_{\alpha^{-1}\gamma}(B)\leq \tfrac{1}{2}\sum_{\alpha\in\overline{K}_{\rm s}^\times\setminus K^\times}\iota_\alpha(A)\sum_{\beta\in\overline{K}_{\rm s}^\times\setminus K^\times}\iota_\beta(B).\]
Combining the different estimates proves $\iota(A\otimes B)\leq \frac{1}{2}$.
\end{proof}

\begin{proposition}[\cite{arzhantsevapaunescu13}, Proposition~5.3]\label{prop:ArzhantsevaPaunescuProposition5.3}
 Let $f\colon [\frac{1}{2},1]\rightarrow [\frac{1}{2},1]$ be defined by 
\[f(x)\coloneqq x^2+(1-x)^2.\]
 Then $f$ is a strictly monotone increasing bijection and $x\in[\frac{1}{2},1[$ implies 
\[\lim_{m\rightarrow \infty}f^m(x)=\tfrac{1}{2}.\]
\end{proposition}

\begin{lemma}\label{lem:RankLengthAndJordanLengthForDirectSumsAndTensorProducts}
 Let $A\in \GL_n(K)$, $B\in \GL_m(K)$ and $\alpha\in K^\times$. Then
\begin{enumerate}[label={\rm(\arabic*)}]
 \item $\lr(A\oplus B)=\frac{n}{n+m}\lr(A)+\frac{m}{n+m}\lr(B)$,\label{enum:RankLengthOfDirectSum}
\item $\lr(A\otimes B)\leq \lr(A)+\lr(B)$,\label{enum:RankLengthOfTensorProduct}
\item $\lj(\alpha\cdot A\oplus B) \leq \min\{\frac{n}{n+m}+\frac{m}{n+m}\lj(B),\frac{n}{n+m}\lj(A)+\frac{m}{n+m}\}$\\
  and $\lj((\alpha A)\oplus B)=\frac{n}{n+m}\lj(A)+\frac{m}{n+m}\lj(B)$,\\
 if $\lj(A)=\lr(\alpha\beta A)$ and $\lj(B)=\lr(\beta B)$ for some $\beta\in K^\times$,\label{enum:JordanLengthOfDirectSum}
\item $\lj(\alpha\cdot A\otimes B)\leq \lj(A)+\lj(B)$.\label{enum:JordanLengthOfTensorProduct}
\end{enumerate}
 \end{lemma}

\begin{proof}
Equation~\ref{enum:RankLengthOfDirectSum} follows from
\[\rank(1-A\oplus B)=\rank(1-A)+\rank(1-B).\]

The matrix $A\otimes B$ acts on $K^{nm}$ by $A\otimes B (v\otimes w)= A(v)\otimes B(w)$ for all $v\in K^n$, $w\in K^m$ and linear extension. Therefore $A(v)=v$ and $B(w)=w$ implies $A\otimes B(v\otimes w)=v\otimes w$, whence $\dim\ker(1-A\otimes B)\geq \dim\ker(1-A)\cdot \dim\ker(1-B)$.
Let $n_0\coloneqq \dim\ker(1-A)$ and $m_0\coloneqq \dim\ker(1-B)$. Then
\begin{align*}
 2\frac{nm-n_0m_0}{nm}&=\frac{nm-nm_0+nm_0-n_0m_0}{nm}+\frac{nm-mn_0+mn_0-n_0m_0}{nm}\\
&=\frac{m-m_0}{m}+\frac{m_0(n-n_0)}{nm}+\frac{n-n_0}{n}+\frac{n_0(m-m_0)}{nm}\\
&=\left(1+\frac{m_0}{m}\right)\frac{n-n_0}{n}+\left(1+\frac{n_0}{n}\right)\frac{m-m_0}{m}\\
&\leq 2\frac{n-n_0}{n}+2\frac{m-m_0}{m}
\end{align*}
implies $\lr(A\otimes B)\leq \lr(A)+\lr(B)$.

There are $\beta_1,\beta_2\in K^\times$ such that $\lj(A)=\lr(\alpha \beta_1 A)$ and $\lj(B)=\lr(\beta_2 B)$. If $\beta_1=\beta_2$, then $\lj((\alpha A)\oplus B)=\frac{n}{n+m}\lj(A)+\frac{m}{n+m}\lj(B)$ follows from \ref{enum:RankLengthOfDirectSum} and invariance of $\lj$ under scalar multiplication. Otherwise by the definition of $\lj$ as an infimum only the inequality in \ref{enum:JordanLengthOfDirectSum} holds.

At last \ref{enum:JordanLengthOfTensorProduct} follows from \ref{enum:RankLengthOfTensorProduct} by linearity of the tensor product and invariance of $\lj$ under scalar multiplication.
\end{proof}

Let $A$ be a matrix in $\GL_n(K)$. We write
\[A^{\otimes k} \coloneqq A\otimes \ldots \otimes A,\]
where $A$ appears $k$ times on the right side.

The construction in the next theorem is essentially from \cite{arzhantsevapaunescu13}, Theorem~5.10. The proof is modified, though, to fit our treatment using almost homomorphisms, to work in arbitrary characteristic and when approximating with the Jordan length.

\begin{theorem}\label{thm:AmplificationInMatrixGroups}
 Let $K$ be a field. The class of groups $\set{\GL_n(K)}{n\in \mathbb{N}}$ with rank length or Jordan length has the amplification property. In particular, for every projectively $K$-sofic group $\Gamma$, finite subset $E\subset \Gamma$ and $\epsilon>0$ there exists an $(E,\epsilon)$-homomorphism $\phi$ satisfying $\lj(\phi(g))>\frac{1}{4}-\epsilon$ for all $g\in E$. In the case of $\Gamma$ being $K$-sofic we can achieve the analogous estimate $\lr(\phi(g))>\frac{1}{8}-\epsilon$.
\end{theorem}

\begin{proof}
We will first treat the case of approximation with the Jordan length. Consider any projectively $K$-sofic group $\Gamma$, $\epsilon>0$ and a finite subset $E\subset \Gamma$, and let $\delta\coloneqq\min_{g\in E}\delta_g$. By Proposition~\ref{prop:ArzhantsevaPaunescuProposition5.3} there is a natural number $m$ such that $f^m(1-\delta)\leq \frac{1}{2}+2\epsilon$, where $f(x)\coloneqq x^2+(1-x)^2$. Choose $\epsilon'>0$ such that $\epsilon'<2^{-m}\epsilon$. Then there exists an $(E,\epsilon')$-homomorphism $\phi\colon \Gamma\rightarrow \GL_n(K)$, where $\lj(\phi(g))\geq \delta$ for all $g\in E$ and we can assume $\phi(1)=1$.
Recursively define
\[\phi_1(g)\coloneqq \phi(g),\quad\quad \phi_{k+1}(g)\coloneqq \phi_k(g)\otimes \phi_k(g).\]
We shall prove that $\phi_m\colon \Gamma\rightarrow \GL_{n^{2^m}}(K)$ is an $(E,\epsilon)$-homomorphism satisfying $\lj(\phi(g))\geq \frac{1}{4}-\epsilon$ for all $g\in E$. If $g\in E$, then $\iota(\phi(g))\leq 1-\lj(\phi(g))\leq 1-\delta$. Therefore $\iota(\phi_k(g))\leq f^k(\iota(\phi(g)))$ as long as $\iota(\phi_{k-1}(g))\geq \frac{1}{2}$. If $\iota(\phi_{k-1}(g))<\frac{1}{2}$ for one $k\leq m$, then $\iota(\phi_k(g))<f(\iota(\phi_{k-1}(g)))<\frac{1}{2}$ by Lemma~\ref{lem:OneByOneJordanBlocksOfTensorProducts}. Otherwise by the choice of $m$ still $\iota(\phi_m(g))\leq \frac{1}{2}+2\epsilon$. Hence in any case 
\[\lj(\phi_m(g))\geq \tfrac{1}{2}(1-\iota(\phi_m(g)))\geq \tfrac{1}{4}-\epsilon.\]
Furthermore Lemma~\ref{lem:RankLengthAndJordanLengthForDirectSumsAndTensorProducts} implies
\begin{align*}
&\lj(\phi_m(g)\phi_m(h)\phi_m(gh)^{-1})\\
&\quad\quad =\lj((\phi(g)\phi(h)\phi(gh)^{-1})^{\otimes 2^m})\\
&\quad\quad \leq 2^m\lj(\phi(g)\phi(h)\phi(gh)^{-1})\leq \epsilon,
\end{align*}
whenever $g,h\in E$.

Now suppose we are approximating with the rank length. To the pair $(E,\epsilon)$ choose $m$ such that $f^m(1-\delta)\leq \frac{1}{2}+4\epsilon$, and $\epsilon'<2^{-m}\epsilon$. Then there is an $(E,\epsilon')$-homomorphism $\phi$ such that $\lr(\phi(g))\geq \delta$ for all $g\in E$. We define $\phi_m$ as before, and additionally
\[\chi_k(g)\coloneqq \phi(g)\otimes \id_{n^{2^k}},\quad\quad \psi_k(g)\coloneqq \phi_k(g)\oplus \chi_k(g).\]

If $\iota(\phi(g))\leq 1-\delta$, then we proceed as above to deduce $\lr(\phi_m(g))\geq \lj(\phi_m(g))\geq \frac{1}{4}-2\epsilon$, and hence $\lr(\psi_m(g))\geq \frac{1}{8}-\epsilon$. If $\iota(\phi(g))>1-\delta$, then $\iota(\phi(g))=\iota_1(\phi(g))$ is impossible, because this would imply $\lr(\phi(g))\leq 1-\iota_1(\phi(g))<\delta$, contrary to the hypothesis. Thus we can assume $\iota(\phi(g))=\iota_\alpha(\phi(g))$, where $\alpha\neq 1$. In this case $\iota_1(\phi(g))\leq 1-\iota_\alpha(\phi(g))<\delta$ and so $\lr(\chi_m(g))=\lr(\phi(g))>\frac{1}{2}(1-\delta)$. At the same time $\lr(\chi_m(g))>\delta$, which implies $\lr(\chi_m(g))\geq \frac{1}{4}$. Thus $\lr(\psi_m(g))\geq \frac{1}{8}$ follows. Showing $\lj(\psi_m(g)\psi_m(h)\psi_m(gh)^{-1})\leq \epsilon$ for $g,h\in E$ works as before.
\end{proof}

Note that Proposition~5.13 in \cite{arzhantsevapaunescu13}, which is a special case of the previous theorem for matrices over $\mathbb{C}$ and the rank length, works with $\frac{1}{4}-\epsilon$ instead of $\frac{1}{8}-\epsilon$.

We will use the amplification properties proved in the previous theorem in \ref{sec:LinearlySoficAndProjectivelyLinearlySoficGroups}.

\section{Linear group approximation}\label{sec:GroupApproximationWithMatrices}
Since we are in particular interested in finite matrix groups, we will use the abbreviation (\deph{projectively}) $q$\deph{-sofic} instead of (projectively) $\mathbb{F}_q$-sofic. If $\Gamma$ can be embedded in an ultraproduct of groups $\GL_{n_i}(K_i)$ with respect to the Jordan length, where the $K_i$ are finite fields of characteristic $p_i$ and moreover $\lim_{\mathfrak{u}}p_i=\infty$, then we call $\Gamma$ \deph{projectively} $0$\deph{-sofic}. If the rank length is used instead we will call $\Gamma$ $0$\deph{-sofic}. We will also write (\deph{projectively}) \deph{prime sofic} or (\deph{projectively}) \deph{zero sofic} if $\Gamma$ is (projectively) $q$-sofic, and $q$ is a prime or $0$, respectively. In the following we will examine the interplay of approximation with the rank length and Jordan length, and approximation in matrix groups over different fields.

The proof of the next proposition is clear.

\begin{proposition}\label{prop:PrimeSoficForInfinitelyManyPrimes}
 Let $\Gamma$ be a (projectively) $p$-sofic group for infinitely many primes $p$. Then $\Gamma$ is (projectively) zero sofic.
\end{proposition}

A metric ultraproduct of symmetric groups $\mup{\prod_{i\in I} S_{n_i}}{\mathfrak{u}}$ is called a \deph{universal sofic} group. We define \deph{universal (projectively) linearly sofic} groups as ultraproducts of groups $\GL_{n_i}(K_i)$ equipped with the rank length (or Jordan length). Hence by Theorem~\ref{thm:ApproximationPropertyAndMetricUltraproducts} we deduce the following statement.

\begin{proposition}\label{prop:ProjectivelyLinearlySoficGroupsAndUltraproducts}
 Let $\Gamma$ be a group. Then $\Gamma$ is (projectively) linearly sofic if and only if it embeds into a universal (projectively) linearly sofic group.
\end{proposition}

The whole terminology of group approximation with matrices we introduced so far can be used accordingly to define respective universal groups. (For an instance, a universal zero sofic group would be $\mup{\prod \GL_{n_i}(K_i)}{\mathfrak{u}}$, where $K_i$ is finite of characteristic $p_i$ and $\lim_{\mathfrak{u}}p_i=\infty$.)
Proposition~\ref{prop:ProjectivelyLinearlySoficGroupsAndUltraproducts} can be reformulated for such a more restrictive setup.

The following proposition is not hard to prove. (Confer also Theorem~1.3 in \cite{arzhantsevapaunescu13}.)

\begin{proposition}\label{prop:SoficImpliesLinearlySofic}
Let $K$ be any field. Every sofic group is (projectively) $K$-sofic.
\end{proposition}

Combining Proposition~\ref{prop:PrimeSoficForInfinitelyManyPrimes} and Proposition~\ref{prop:SoficImpliesLinearlySofic}, we see that the class of sofic groups provides many examples of groups that are simultaneously prime sofic and zero sofic.

\begin{lemma}\label{lem:DirectProductsOfProjectivelyLinearlySoficGroups}
 Let $\Gamma_1$ and $\Gamma_2$ be groups, $E_i\subset \Gamma_i$ finite, and $\phi\colon \Gamma_1\rightarrow \GL_n(K)$ and $\psi\colon \Gamma_2\rightarrow \GL_m(K)$ be $(E_i,\frac{1}{2}\epsilon)$-homomorphisms with respect to the Jordan length for $i=1,2$. Then
\[\zeta_{(g,h)}\coloneqq \phi_g\otimes \psi_{h}\in \GL_{nm}(K)\]
defines an $(E_1\times E_2,\epsilon)$-homomorphism on $\Gamma_1\times \Gamma_2$.
\end{lemma}

\begin{proof}
 Assume $\lj(\phi_g)\geq \delta >0$ and $\lj(\psi_h)\geq \delta$ for all $g\in E_1$, $h\in E_2$. Note that \ref{enum:JordanLengthOfDirectSum} in Lemma~\ref{lem:RankLengthAndJordanLengthForDirectSumsAndTensorProducts} is in the general form a very weak estimate, compared to \ref{enum:RankLengthOfDirectSum}. This is the reason why we are working with tensor products. By Lemma~\ref{lem:RankLengthAndJordanLengthForDirectSumsAndTensorProducts} for $g,g'\in E_1$ and $h,h'\in E_2$
\begin{align*}
&\lj(\zeta_{(g,h)}\zeta_{(g',h')}\zeta_{(g,h)(g',h')}^{-1})\\
&\quad\quad =\lj(\phi_g\phi_{g'}\phi_{gg'}^{-1}\otimes \psi_h\psi_{h'}\psi_{hh'}^{-1})\\
&\quad\quad \leq \lj(\phi_g\phi_{g'}\phi_{gg'}^{-1})+\lj(\psi_h\psi_{h'}\psi_{hh'}^{-1})\\
&\quad\quad \leq \tfrac{1}{2}\epsilon+\tfrac{1}{2}\epsilon=\epsilon.
\end{align*}
Assume one of $\iota(\phi_g)$ and $\iota(\psi_h)$ is larger than $\frac{1}{2}$. Then we use Proposition~\ref{prop:LengthFunctionsAndIotaValuesEstimates} and Lemma~\ref{lem:OneByOneJordanBlocksOfTensorProducts} to estimate
\begin{align*}
 \lj(\zeta_{(g,h)}) &=\lj(\phi_g\otimes \psi_h)\\
&\geq \tfrac{1}{2}(1-\iota(\phi_g\otimes \psi_h))\\
&\geq \tfrac{1}{2}(1-\iota(\phi_g)\iota(\psi_h)-(1-\iota(\phi_g))(1-\iota(\psi_h)))\\
&= \tfrac{1}{2}\iota(\phi_g)(1-\iota(\psi_h))+\tfrac{1}{2}\iota(\psi_h)(1-\iota(\phi_g)).
\end{align*}
Hence $\lj(\zeta_{(g,h)})\geq \tfrac{1}{4}\lj(\phi_g)$ or $\lj(\zeta_{(g,h)})\geq \tfrac{1}{4}\lj(\psi_h)$, which is large enough.
 If $\iota(\phi_g),\iota(\psi_h)\leq \frac{1}{2}$, then, also by 
 Lemma~\ref{lem:OneByOneJordanBlocksOfTensorProducts}, 
\[\lj(\zeta_{(g,h)}) \geq \tfrac{1}{2}(1-\iota(\phi_g\otimes \psi_h))\geq \tfrac{1}{4}.\]
Thus $\zeta$ is an $(E_1\times E_2,\epsilon)$-homomorphism.% This shows that the direct product of projectively $K$-sofic groups is projectively $K$-sofic.
\end{proof}

\begin{theorem}\label{thm:EquivalenceOfProjectivelyLinearlyAndLinearlySofic}
A group $\Gamma$ is $K$-sofic if and only if it is projectively $K$-sofic.
\end{theorem}

\begin{proof}
Let $\Gamma$ be a $K$-sofic group and $\phi\colon \Gamma\rightarrow \GL_n(K)$ an $(E,\epsilon)$-homomorphism with respect to the rank length, where we may assume $\epsilon<\frac{1}{2}$. Let 
\[\psi_g\coloneqq \phi_g\oplus \id_n\in \GL_{2n}(K).\]
 Then for every $g\in \Gamma $ evidently $\lr(\psi_g)\leq \frac{1}{2}$. By Corollary~2.14 in \cite{stolzthom13} $\lj(\psi_g)=\lr(\psi_g)$ follows. Hence $\lj(\psi_g)> \frac{\delta}{2}$ is true for all $g\in E$. Also, since $\epsilon<\frac{1}{2}$, by Corollary~2.14, \cite{stolzthom13} we have 
\[\lj(\phi_g\phi_h\phi_{gh}^{-1})=\lr(\phi_g\phi_h\phi_{gh}^{-1})\]
 and consequently
\[\lj(\psi_g\psi_h\psi_{gh}^{-1})<\tfrac{1}{2}\epsilon,\]
whenever $g,h\in E$. Thus $\Gamma$ is projectively $K$-sofic.

Suppose conversely that $\Gamma$ is projectively $K$-sofic. Let $A^{\rm T}$ denote the transpose of a matrix $A$, and $A^{-\rm T}\coloneqq (A^{\rm T})^{-1}$. Then it is easily seen that $\lr(A^{-\rm T})=\lr(A)$ and $\lj(A^{-\rm T})=\lj(A)$.
We choose an $(E,\frac{1}{2}\epsilon)$-homomorphism $\phi$, where $\epsilon<\frac{1}{4}$ and define 
\[\psi_g\coloneqq \phi_g\otimes \phi_g^{-\rm T}.\]
If $\alpha$ is an eigenvalue of $\phi_g$, then clearly $\alpha^{-1}$ is an eigenvalue of $\phi_g^{-\rm T}$. Therefore $g\mapsto \phi_g^{-\rm T}$ defines an $(E,\frac{1}{2}\epsilon)$-homomorphism. Then by embedding $\Gamma$ diagonally into $\Gamma\times \Gamma$ Lemma~\ref{lem:DirectProductsOfProjectivelyLinearlySoficGroups} shows that $\psi$ is an $(E,\epsilon)$-homomorphism with respect to the Jordan length. By the choice of $\epsilon$, the prevalent eigenvalue of $\psi_g\psi_h\psi_{gh}^{-1}$ is $1$ for all $g,h\in E$. Therefore we have 
\[\lr(\psi_g\psi_h\psi_{gh}^{-1})=\lj(\psi_g\psi_h\psi_{gh}^{-1})\leq \epsilon.\]
 Moreover $\lr(\psi_g)\geq \lj(\psi_g)$ for $g\in E$. This shows that $\Gamma$ is $K$-sofic.
\end{proof}

Theorem~\ref{thm:EquivalenceOfProjectivelyLinearlyAndLinearlySofic} will be of great use to reduce problems concerning projectively linearly sofic groups to linearly sofic groups. For example the proof of Theorem~\ref{thm:AmplificationInMatrixGroups} could be done for linearly sofic groups and the statement for projectively sofic groups derived from Theorem~\ref{thm:EquivalenceOfProjectivelyLinearlyAndLinearlySofic}. Of course this kind of reduction also works the other way round, but usually linearly sofic groups are a bit easier to handle.

The following proposition asserts that linear soficity is preserved when passing to larger fields. It is a good example of an application of Theorem~\ref{thm:EquivalenceOfProjectivelyLinearlyAndLinearlySofic}. The converse statement is more complicated and will be explained afterwards.

\begin{proposition}\label{prop:ApproximationWithMatricesInLargerFields}
Let $K$ be a field. If $\Gamma$ is $K$-sofic, then it is $L$-sofic for any field $L$ containing $K$.
\end{proposition}

\begin{proof}
The claim for linearly sofic group follows directly from the definition of the rank length.
\end{proof}

The proof of the following lemma is obtained by means of linear algebra.

\begin{lemma}\label{lem:MatrixActionOnFieldExtensions}
 Let $L/K$ be a finite field extension. Then matrices $A\in \GL_n(L)$ act as invertible linear transformations $A'$ on $K^{n\cdot [L\colon K]}$ and 
\[\dim_K\ker(1-A')= [L\colon K]\cdot \dim_L\ker(1-A)\]
holds for all $A\in\GL_n(L)$.
\end{lemma}

\begin{theorem}\label{thm:MatrixApproximationOnAlgebraicFieldExtensions}
Let $L/K$ be an algebraic field extension and suppose $\Gamma$ is $L$-sofic. Then $\Gamma$ is $K$-sofic.
\end{theorem}

\begin{proof}
Suppose $\Gamma$ is $L$-sofic. We consider an $(E,\epsilon)$-homomorphism $\phi\colon \Gamma\rightarrow \GL_n(L)$. Let $M$ be the set of all entries of matrices in $\phi(E^2)$. Then $K(M)$ is a subfield of $L$ and an algebraic extension of $K$ of finite degree $m$. By Lemma~\ref{lem:MatrixActionOnFieldExtensions} the elements of $\phi(E^2)$ act on $K^{nm}$ as linear transformations and their rank length remains unchanged. Thus we obtain an $(E,\epsilon)$-homomorphism $\phi\colon \Gamma\rightarrow \GL_{nm}(K)$, if we use this action for elements in $\phi(E^2)$ and some arbitrary extension to $\Gamma\setminus E^2$. Hence $\Gamma$ is $K$-sofic.
\end{proof}

Let $L/K$ be a field extension. Let $x=(x_1,\ldots,x_m)$ be a vector in $L^m$. We say that $x'=(x'_1,\ldots ,x'_m)$ is a \deph{specialization} of $x$ over $K$ if every polynomial $f\in K[X_1,\ldots ,X_m]$ vanishing at $x$ does also vanish at $x'$. Moreover $x'$ is an \deph{algebraic} specialization if $K(x')/K$ is an algebraic extension.

The next theorem is also generalizes results in \cite{arzhantsevapaunescu13} but was obtained independently.

\begin{theorem}\label{thm:MatrixApproximationOnFieldExtensions}
Let $L/K$ be a field extension and suppose $\Gamma$ is $L$-sofic. Then $\Gamma$ is $K$-sofic.
\end{theorem}

\begin{proof}
We consider an $(E,\epsilon)$-homomorphism $\phi\colon \Gamma\rightarrow \GL_n(L)$. Let $X$ be the set of all entries of matrices in $\phi(E^2)$. For every $g\in E$ let $a_g$ be a maximal $k_g\times k_g$-submatrix of $1-\phi_g$ such that $\det(a_g)\neq 0$, i.e. $k_g=\rank(1-\phi_g)$. Then there is $\alpha_g\in L$ such that $\alpha_g\det(a_g)-1=0$. Let $Y$ be the set of all $\alpha_g$ for $g\in E$. We order the elements of $X\cup Y$ as a vector $x$. Then by Theorem~7 in \cite{lang73}, Chapter~II, there exists an algebraic specialization $x'$ of $x$.
This implies that $K(x')$ is an algebraic extension of $K$. We write $\psi_g$ for the matrix in $K(x')$ obtained by replacing elements in $x$ with appropriate elements in $x'$, and $a'_g$ for the submatrix of $1-\psi_g$ corresponding to $a_g$. Then every submatrix of $1-\psi_g$ larger than $a'_g$ has zero determinant, and $\alpha'_g\det(a'_g)-1=0$, since the determinant is a polynomial in matrix entries. Therefore $\det(a'_g)\neq 0$, $\rank(1-\psi_g)=k_g$ and consequently $\lr(\psi_g)=\lr(\phi_g)$. By the same reasoning $\lr(\psi_g\psi_h\psi_{gh}^{-1})\leq \epsilon$. If we define $\psi_g$ arbitrary for $g\notin E^2$, then $\psi$ is an $(E,\epsilon)$-homomorphism into $\GL_n(K(x'))$. We have thus shown that $\Gamma$ is $K(x')$-sofic, and since $K(x')/K$ is algebraic, by Theorem~\ref{thm:MatrixApproximationOnAlgebraicFieldExtensions} $\Gamma$ is $K$-sofic.
\end{proof}

\begin{corollary}
Let $q=p^k$, where $p$ is a prime power. Then every $q$-sofic group is $p$-sofic.
\end{corollary}

\begin{corollary}\label{cor:FiniteMatrixApproximationImpliesPrimeZeroSofic}
 Let $\Gamma$ have the approximation property in $\GLinG_{\fin}$ with the rank length. Then $\Gamma$ is $p$-sofic for $p$ a prime or $0$.
\end{corollary}

\begin{proof}
 We can embed $\Gamma$ into a metric ultraproduct of groups $\GL_{n_i}(K_i)$, where the $K_i$ are finite fields. If $K_i$ has characteristic $p_i$ and $\lim_{\mathfrak{u}}p_i=\infty$, then $\Gamma$ is by definition $0$-sofic. If otherwise $p_i\leq C\alev{\mathfrak{u}}$ for some constant $C$, then, because $\mathfrak{u}$ is an ultrafilter, $p_i=p\alev{\mathfrak{u}}$. By a standard ultraproduct argument $\mup{\prod \GL_{n_i}(K_i)}{\mathfrak{u}}$ is isomorphic to $\mup{\prod_{i\colon \chr(K_i)=p}\GL_{n_i}(K_i)}{\mathfrak{u}}$. By Theorem~\ref{thm:MatrixApproximationOnAlgebraicFieldExtensions} we can replace all $K_i$ of characteristic $p$ with the field $\mathbb{F}_p$ to show that $\Gamma$ is $p$-sofic.
\end{proof}

In \cite{arzhantsevapaunescu13}, Theorem~8.2 it was proved that a $\mathbb{C}$-sofic group is prime sofic or zero sofic. We adopt exactly the same argument to obtain for arbitrary fields:

\begin{theorem}\label{thm:MatrixApproximationOnArbitraryFieldExtensions}
Let $\Gamma$ be linearly sofic. Then $\Gamma$ has the $\GLinG_{\fin}$-approximation property with respect to the rank length.
\end{theorem}

\begin{corollary}
 Every linearly sofic group is $p$-sofic, where $p$ is a prime or $p=0$.
\end{corollary}

We are now confronted with the two similar Theorems~\ref{thm:MatrixApproximationOnFieldExtensions} and \ref{thm:MatrixApproximationOnArbitraryFieldExtensions}. The latter is somewhat stronger insofar as it reduces approximation to finite fields, but it lacks the virtue of the former of preserving the field characteristic. 

We are left with some open questions. The most important one in this context surely is whether linearly sofic groups are sofic. It is also unclear if $p$-sofic groups are $q$-sofic for different primes $p$ and $q$, or if prime sofic or zero sofic groups are $\mathbb{Q}$-sofic. These questions were also adressed in \cite{arzhantsevapaunescu13}, Question~8.6. A positive answer combined with Theorem~7.4, ibid. would for example solve Kaplansky's Direct Finiteness Conjecture for linearly sofic groups (see also ibid, Question~7.9.)

\section{The class of linearly sofic groups}\label{sec:LinearlySoficAndProjectivelyLinearlySoficGroups}
In \cite{elekszabo06} it was proved that the class of sofic groups is closed under taking subgroups, direct products, direct limits, inverse limits, free products and extensions by amenable groups. Later Collins and Dykema proved in \cite{collinsdykema11} that free products of sofic groups amalgamated over monotileably amenable groups are sofic. The unpleasant restriction of monotileability was removed by P\u{a}unescu in \cite{paunescu11} and Elek and Szab\'o in \cite{elekszabo11}. We shall show that similar conclusions as in \cite{elekszabo06} are true for linearly sofic groups, but we will not go as far as treating amalgamated products.

\begin{proposition}\label{prop:SubgroupsDirectProductsAndLimits}
 Let $K$ be a field. The class of  $K$-sofic groups is closed under taking subgroups, inverse limits, direct products and direct limits.
\end{proposition}

\begin{proof}
The claim concerning subgroups and inverse limits is an immediate consequence of Proposition~\ref{prop:SubgroupsAndInverseLimitsPermanence}.

Now suppose $\Gamma$ is a direct limit of $K$-sofic groups $\Gamma_i$. By Proposition~\ref{prop:DirectLimitsPermanence} and Theorem~\ref{thm:AmplificationInMatrixGroups} $\Gamma$ is $K$-sofic.

By Lemma~\ref{lem:RankLengthAndJordanLengthForDirectSumsAndTensorProducts} $\GL_n(K)\times \GL_m(K)$ embeds isometrically into $\GL_{n+m}(K)$ when using the mapping $(A,B)\mapsto A\oplus B$. Thus, with an appeal to induction, Proposition~\ref{prop:DirectProductsPermanence} shows that direct products of $K$-sofic groups are $K$-sofic.
\end{proof}

\begin{lemma}\label{lem:LengthOfTensorProductWithPermutationMatrix}
 Let $K$ be a field and $\pi\in S_n$ a permutation without fixed points. Let $A_i\in \GL_m(K)$ for $i=1\ldots n$, and $P_\pi$ be the permutation matrix in $\GL_n(K)$ corresponding to $\pi$. We define $A\in \GL_{nm}(K)$ by $A(e_i\otimes e_j)\coloneqq e_i\otimes A_i(e_j)$ and linear extension. Then $\rank(1-(P_\pi\otimes \id_m)\circ A)\geq \frac{1}{2}nm$.
\end{lemma}

\begin{proof}
The matrix $P_\pi\otimes \id_m$ is blockdiagonal, where every block corresponds to a cycle of $\pi$. Hence it suffices to assume that $\pi$ consists of a single cycle. Then it is an elementary observation that if $n$ is even, $1-(P_\pi\otimes \id_m) \circ A$ has a submatrix $\id_{\frac{1}{2}nm}$. If $n$ is odd there is a $\frac{1}{2}(n+1)m\times \frac{1}{2}(n+1)m$-submatrix which is a block matrix having blocks $\id_m$ on the diagonal and one block $-A_j$ for some $j$. In both cases the determinant of the submatrix is non-zero and the claim follows.
\end{proof}

The following theorem is motivated by Item~3 in Theorem~1, \cite{elekszabo06}. A variant of the statement appears in \cite{arzhantsevapaunescu13}, Theorem~9.3 for $\mathbb{C}$-sofic groups, the proof of which can be easily generalized to fit our situation.

\begin{theorem}
 Let $K$ be a field and $\Gamma$ a group such that $N\lhd \Gamma$ is $K$-sofic and $\Gamma/N$ is amenable. Then $\Gamma$ is $K$-sofic.
\end{theorem}

\begin{lemma}\label{lem:AlmostHomomorphismsAdaptedToInverses}
 Let $\Gamma$ be a group with the $\mathscr{G}$-approximation property. Then for any finite subset $1\in E\subset\Gamma$ and $\epsilon>0$ there is an $(E,\epsilon)$-homomorphism $\phi\colon \Gamma\rightarrow G\in\mathscr{G}$ such that $\phi(1)=1$ and  $\phi(g^{-1})=\phi(g)^{-1}$ for all $g\in E$ not of order $2$.
\end{lemma}

\begin{proof}
 Let $F\coloneqq (E\cup E^{-1}\cup\{1\})^2$ and $\psi\colon \Gamma \rightarrow G$ be an $(F,\frac{1}{2}\epsilon)$-homomorphism into a group $G\in\mathscr{G}$ with invariant pseudo length function $\ell$.
By Proposition~\ref{prop:RelaxedAlmostHomomorphisms} without loss of generality $\psi(1)=1$. We partition $F$ into three subsets as follows: Let $F_0$ be the set of all $g\in F$ of order $2$ or $g=1$. Then we partition $F\setminus F_0$ into $F_1$ and $F_{-1}$ such that $g\in F_{-1}$ implies $g^{-1}\in F_1$, or equivalently $g\in F_1$ implies $g^{-1}\in F_{-1}$.
 We define 
\[\phi(g)\coloneqq \left\{\begin{array}{l l}
                                   \psi(g), &g\in F_0\cup F_1,\\
                                   \psi(g^{-1})^{-1}, &g\in F_{-1}.
                                  \end{array}\right.
\]
Then obviously $g\in F_{-1}$ implies $\phi(g)^{-1}=\phi(g^{-1})$, and the same holds for $g\in F_1$, since in this case $g^{-1}\in F_{-1}$.

We must prove that $\phi$ is an $(E,\epsilon)$-homomorphism. It is enough to show that $g,h\in E$ implies $\ell(\phi(g)\phi(h)\phi(gh)^{-1})\leq \epsilon$. The case of $g,h,gh\in F_0\cup F_1$ is clear. The case of $g,h,gh\in F_{-1}$ reduces to
\[\ell(\phi(g)\phi(h)\phi(gh)^{-1})=\ell(\psi(h^{-1})\psi(g^{-1})\psi(h^{-1}g^{-1})^{-1})\leq \tfrac{1}{2}\epsilon,\]
where we used the invariance of $\ell$.
Let $g,h\in F_0\cup F_1$, $gh\in F_{-1}$. Then
\begin{align*}
&\ell(\phi(g)\phi(h)\phi(gh)^{-1})\\
&\quad\quad =\ell(\psi(g)\psi(h)\cdot \psi(gh)^{-1}\psi(gh)\cdot \psi((gh)^{-1}))\\
&\quad\quad \leq \ell(\psi(g)\psi(h)\psi(gh)^{-1})+\ell(\psi(gh)\psi((gh)^{-1}))\\
&\quad\quad \leq \tfrac{1}{2}\epsilon +\tfrac{1}{2}\epsilon=\epsilon.
\end{align*}
If $g\in F_0\cup F_1$ and $h, gh\in F_{-1}$ we estimate
\begin{align*}
&\ell(\phi(g)\phi(h)\phi(gh)^{-1})\\
&\quad\quad =\ell(\psi(g)\cdot \psi(g^{-1})\psi(g^{-1})^{-1}\cdot \psi(h^{-1})^{-1}\psi((gh^{-1}))\\ 
&\quad\quad \leq \ell(\psi(g)\psi(g^{-1}))+\ell(\psi(h^{-1})\psi(g^{-1})\psi(h^{-1}g^{-1})^{-1})\\
&\quad\quad \leq \tfrac{1}{2}\epsilon +\tfrac{1}{2}\epsilon=\epsilon. 
\end{align*}
Finally assume $g, gh\in F_0\cup F_1$ and $h\in F_{-1}$. Then 
\begin{align*}
&\ell(\phi(g)\phi(h)\phi(gh)^{-1})\\
&\quad\quad =\ell(\psi(g)\psi(h^{-1})^{-1}\cdot \psi(h)^{-1}\psi(h)\cdot \psi(gh)^{-1})\\
&\quad\quad \leq \ell(\psi(h^{-1})\psi(h))+\ell(\psi(g)\psi(h)\psi(gh)^{-1})\\
&\quad\quad \leq \tfrac{1}{2}\epsilon +\tfrac{1}{2}\epsilon=\epsilon
\end{align*}
holds. The remaining cases follow analogously.
\end{proof}

\begin{lemma}\label{lem:FreeProductsOfFinitelyGeneratedProjectivelyLinearlySoficGroups}
 Let $K$ be a field and $\Gamma_1$ and $\Gamma_2$ finitely generated $K$-sofic groups. Then the free product $\Gamma_1*\Gamma_2$ is $K$-sofic.
\end{lemma}

\begin{proof} 
Suppose $\Gamma_1$ and $\Gamma_2$ are generated by finite symmetric sets $A$ and $B$, respectively. 
 Let $\phi\colon \Gamma_1\rightarrow \GL_n(K)$ be an $(A^{r},\frac{1}{2}\epsilon)$-homomorphism and $\psi\colon \Gamma_2\rightarrow \GL_m(K)$ a $(B^{r},\frac{1}{2}\epsilon)$-homomorphism. By Lemma~\ref{lem:AlmostHomomorphismsAdaptedToInverses} we can assume without loss of generality that $\phi_1=1$, $\psi_1=1$, and $\phi_{g^{-1}}=\phi_g^{-1}$ and $\psi_{h^{-1}}=\psi_h^{-1}$ for all $g\in A^{2r}$ and $h\in B^{2r}$ not of order $2$. We define a mapping 
$\phi'\colon \Gamma_1\rightarrow \GL_{2n}(K)$ by 
\[\phi'_g\coloneqq \left(\begin{matrix} \phi_g & \\ & \phi_g \end{matrix}\right)\]
 if $g^2\neq 1$ and 
\[\phi'_g\coloneqq \left(\begin{matrix} & \phi_g^{-1}\\ \phi_g &\end{matrix}\right)\]
otherwise. We define $\psi'\colon\Gamma_2\rightarrow \GL_{2m}(K)$ analogously. Note that $\phi'$ and $\psi'$ no longer need to be almost homomorphisms. Nevertheless $\phi'_{g^{-1}}=(\phi'_g)^{-1}$ holds for all $g\in A^{2r}$ and $\psi'_{h^{-1}}=(\psi'_h)^{-1}$ holds for all $h\in B^{2r}$.                          

 Now consider the subgroup of $\GL_{2n}(K)$ generated by $\set{\phi'_g}{g\in A^{2r}}$. By Malcev's Theorem this group is residually finite and hence there is a finite group $H_1$ and a homomorphism 
\[\pi_1\colon\left<\set{\phi'_g}{g\in A^{2r}}\right>\rightarrow H_1,\]
the restriction of which to $\set{\phi'_g}{g\in A^{2r}}$ is injective. Analogously there are $H_2$ and
\[\pi_2\colon\left<\set{\psi'_h}{h\in B^{2r}}\right>\rightarrow H_2.\]
Since the free product of finite groups is residually finite by \cite{gruenberg57}, Theorem~4.1, there is a finite group $G$ and a homomorphism $\pi\colon H_1*H_2\rightarrow G$ such that $\pi(h_1g_1\ldots h_kg_k)\neq 1$ for all reduced words $h_1g_1\ldots h_kg_k$ of length $2k$ in $H_1*H_2$, $h_i\in H_1$ and $g_i\in H_2$, and $k\leq 2r$. For elements $g\in A^{2r}$ or $h\in B^{2r}$ we write 
\[\overline{g}\coloneqq \pi(\pi_1(\phi'_g))\in G,\quad \quad \overline{h}\coloneqq \pi(\pi_2(\psi'_h))\in G.\]
To summarize, $a,a_i\in A^{2r}$ and $b,b_i\in B^{2r}$ imply $\overline{a^{-1}}=\overline{a}^{-1}$ and $\overline{b^{-1}}=\overline{b}^{-1}$, and $\overline{a_1}\overline{b_1}\ldots \overline{a_k}\overline{b_k}\neq 1$, whenever $k\leq 2r$ and $a_i\neq 1$, $b_j\neq 1$ for $i\neq 1$, $j\neq k$.

Consider the vector space
\[V\coloneqq K^G\otimes (K^n\oplus K^m)\]
with the basis of standard vectors $e_x\otimes e_i$, where $x\in G$, $i=1\ldots n+m$.
If $g\in A^{2r}$ we define
\[\tilde{\phi}_g\coloneqq \id_{K^G}\otimes (\phi_g \oplus \id_m)\]
and similarly for $h\in B^{2r}$
\[\tilde{\psi}_h\coloneqq \id_{K^G}\otimes (\id_n \oplus \psi_h).\]
Let $g=a_1b_1\ldots a_kb_k$ be a reduced word in $(A\cup B)^{2r}\setminus \{1\}\subset \Gamma_1*\Gamma_2$, where $a_i\in \Gamma_1$ and $b_i\in \Gamma_2$ for all $i=1\ldots k$. Then $a_1\ldots a_k\in A^{2r}$ and $b_1\ldots b_k\in B^{2r}$. We let
\[\sigma_g(e_x\otimes e_i)\coloneqq e_{\overline{a_1}\overline{b_1}\ldots \overline{a_k}\overline{b_k}x}\otimes e_i.\]
Then $\sigma_g$ commutes with $\tilde{\phi}_a$ and $\tilde{\psi}_b$ for all $g\in (A\cup B)^{2r}$, $a\in A^{2r}$ and $b\in B^{2r}$. Now we define $\zeta_g$ by
\[\zeta_g\coloneqq \sigma_g\circ \tilde{\phi}_{a_1\ldots a_k}\circ \tilde{\psi}_{b_1\ldots b_k}\]
and linear extension. We let $\zeta_1\coloneqq \id_V$ and extend $\zeta\colon g\mapsto \zeta_g$ arbitrarily to the whole of $\Gamma$ to obtain a mapping $\zeta\colon \Gamma\rightarrow \GL(V)$.

Let $g=a_1b_1\ldots a_kb_k$ and $h=c_1d_1\ldots c_ld_l$ be reduced words in $(A\cup B)^r\setminus\{1\}\subset \Gamma_1*\Gamma_2$. We abbreviate $a\coloneqq a_1\ldots a_k$, $b\coloneqq b_1\ldots b_k$, $c\coloneqq c_1\ldots c_l$ and  $d\coloneqq d_1\ldots d_l$. Suppose when multiplying $g$ and $h$ cancellations occur, i.e.
\[gh=a_1b_1\ldots b_{k-s-1}(a_{k-s}c_{1+s})d_{1+s}\ldots c_ld_l\]
or
\[gh=a_1b_1\ldots a_{k-t}(b_{k-t}d_{1+t})c_{2+t}\ldots c_ld_l.\]
Then without loss of generality in the first case $b_k=1$, $a_{k-i}=c_{1+i}^{-1}$ and $b_{k-j}=d_{j}^{-1}$ for all $i=0\ldots s$ and $j=1\ldots s$. Therefore
$\overline{b_k}=1$, $\overline{a_{k-i}}=\overline{c_{1+i}}^{-1}$ and $\overline{b_{k-j}}=\overline{d_{j}}^{-1}$ in $G$. Thus when multiplying $\overline{a_1}\overline{b_1}\ldots \overline{a_k}\overline{b_k}$ and $\overline{c_1}\overline{d_1}\ldots \overline{c_l}\overline{d_l}$, the same cancellations as in $gh$ occur (and maybe more). This means $\sigma_g\sigma_h=\sigma_{gh}$.

Now by the definition of $\zeta$ we readily obtain
\[\zeta_g\zeta_h-\zeta_{gh}=\sigma_{gh}\circ (\tilde{\phi}_a\circ\tilde{\psi}_b\circ \tilde{\phi}_c\circ\tilde{\psi}_d-\tilde{\phi}_{ac}\circ\tilde{\psi}_{bd}).\]
Because $\sigma_{gh}$ has full rank,
\begin{align*}
\lj(\zeta_g\zeta_h-\zeta_{gh})&=\lj(\tilde{\phi}_a\circ\tilde{\psi}_b\circ \tilde{\phi}_c\circ\tilde{\psi}_d-\tilde{\phi}_{ac}\circ\tilde{\psi}_{bd})\\
&=\lj(\phi_a\oplus \psi_b \circ \phi_c\oplus \psi_d -\phi_{ac}\oplus \psi_{bd})\\
&\leq \frac{n}{n+m}\lr(\phi_a\phi_c-\phi_{ac})+\frac{m}{n+m}\lr(\psi_b\psi_d-\psi_{bd}),
\end{align*}
where the inequality follows from \ref{enum:RankLengthOfDirectSum} in Lemma~\ref{lem:RankLengthAndJordanLengthForDirectSumsAndTensorProducts}.
Since $a,c\in A^{r}$ and $b,d\in B^{r}$, the right hand side is less than $\epsilon$.

It is clear that $\sigma_g$ acts as a permutation matrix modulo $K^n\oplus K^m$. If $1\neq g=a_1b_1\ldots a_kb_k$ is a reduced word of length not more than $2r$ in letters from $A\cup B$, by construction $\overline{a_1}\overline{b_1}\ldots \overline{a_k}\overline{b_k}\neq 1$ in $G$. Since the action of $\sigma_g$ is determined by the permutation action of $\overline{a_1}\overline{b_1}\ldots \overline{a_k}\overline{b_k}$ on $G$, we can use Lemma~\ref{lem:LengthOfTensorProductWithPermutationMatrix} to conclude that $\lj(\zeta_g)\geq \frac{1}{2}$.
\end{proof}

\begin{theorem}\label{thm:FreeProductsOfLinearlySoficGroups}
 Let $K$ be a field. Then the free product of $K$-sofic groups is $K$-sofic.
\end{theorem}

\begin{proof}
 Let $\Gamma_1$ and $\Gamma_2$ be $K$-sofic groups. Then $\Gamma_i*\Gamma_2$ is a direct limit of groups $\Gamma_1^{(i)}*\Gamma_2^{(i)}$, where $\Gamma_j^{(i)}$ is a finitely generated subgroup of $\Gamma_j$ for every $i$ and $j=1,2$. As subgroups of $\Gamma_j$, the groups $\Gamma_j^{(i)}$ are  $K$-sofic, and since they are finitely generated, by Lemma~\ref{lem:FreeProductsOfFinitelyGeneratedProjectivelyLinearlySoficGroups} $\Gamma_1^{(i)}*\Gamma_2^{(i)}$ is  $K$-sofic for all $i$. At last Proposition~\ref{prop:SubgroupsDirectProductsAndLimits} shows that the direct limit $\Gamma_1*\Gamma_2$ is $K$-sofic.
\end{proof}

\small

\vspace{0.5cm}

\textit{Abel Stolz, Universit\"at Leipzig, Augustusplatz 10, 04109 Leipzig, Germany \\}
\verb|abel.stolz@math.uni-leipzig.de|\\

\end{document}